\documentclass[12pt,a4paper]{amsart}
\usepackage{amsmath}
\usepackage{latexsym}
\usepackage{amssymb}
\usepackage{cite}
\usepackage{url}
\usepackage{hyperref}
\usepackage{verbatim}
\usepackage[ruled,boxed,commentsnumbered,norelsize]{algorithm2e}

\newcounter{enu}
\newcounter{enubis}
\newtheorem{theorem}{Theorem}
\newtheorem{definition}[theorem]{Definition}
\newtheorem{lemma}[theorem]{Lemma}
\newtheorem{example}[theorem]{Example}
\newtheorem{problem}[theorem]{Problem}
\newtheorem{proposition}[theorem]{Proposition}

\newtheorem{corollary}[theorem]{Corollary}

\theoremstyle{remark}
\newtheorem{remark}[theorem]{Remark}
\newenvironment{mylist}%
{\begin{list}{P--\arabic{enu}}{\usecounter{enu}}\setcounter{enu}{\theenubis}}%
{\setcounter{enubis}{\theenu}\end{list}}
\newenvironment{barray}[1]{\left[\begin{array}{*{#1}{c}}}{\end{array}\right]}

\newcommand{\Q}{\mathord{\mathbb Q}}
\newcommand{\G}{\mathord{\mathbb G}}
\newcommand{\F}{\mathord{\mathbb F}}
\newcommand{\Z}{\mathord{\mathbb Z}}

\newcommand{\nm}{\mathord{\mathrm N}}

\date{\today}

\title{Continuants and some decompositions into squares}

\author{
Charles Delorme}

\email{\texttt{cd@lri.fr}}

\author{
Guillermo Pineda-Villavicencio}
\address{Centre for Informatics and Applied Optimisation, Federation University Australia}

\email{\texttt{work@guillermo.com.au}}
\begin{document}
\maketitle
\begin{abstract}
In 1855 H.~J.~S. Smith \cite{CELV99} proved Fermat's two-square theorem using the notion of
 palindromic continuants. In his paper, Smith constructed a proper representation of a prime number $p$ as a sum of two squares, given a solution of $z^2+1\equiv0\pmod{p}$,
and vice versa. In this paper, we extend the use of continuants to proper representations by sums of two squares in rings
of polynomials over fields of characteristic different from 2. New deterministic algorithms for finding the corresponding proper representations are presented.

Our approach will provide a new constructive proof of the four-square theorem and new proofs for other representations of integers by quaternary quadratic forms.
\end{abstract}



\section{Introduction}
 Fermat's two-square theorem has always captivated the mathematical community. Equally captivating are the known proofs of such a theorem; see, for instance,
\cite{Zag90,Her1848,Ser1848,CELV99,Bri72, Bar72}.
Among these proofs we were enchanted by Smith's elementary approach
\cite{CELV99}, which is well within the reach of undergraduates.
We remark Smith's proof is very similar to Hermite's \cite{Her1848},
Serret's \cite{Ser1848}, and Brillhart's \cite{Bri72}.

Two main ingredients of Smith's proof are the notion of continuant
(Definition~\ref{def:continuant} for arbitrary rings)
and the famous Euclidean algorithm.

Let us recall here, for convenience, a definition taken from \cite[pp.~148]{Jac85I}
\begin{definition}
\emph{Euclidean rings} are rings $R$ with no zero divisors which are endowed with a
Euclidean function $\nm$  from $R$ to  the nonnegative integers such that
for all $\tau_1, \tau_2\in R$ with $\tau_1\ne 0$, there exists $q, r \in R$
such that $\tau_2 = q\tau_1 + r$ and $\nm(r) < \nm(\tau_1)$.
\end{definition}
Among well-known examples, we are going to
use  the integers with $\nm(u)=|u|$, and polynomials over
a field  with $\nm(P)=2^{\mathrm{degree}(P)}$ and $\nm(0)=0$.

\begin{definition}[Continuants in arbitrary rings, { \cite[Sec.~6.7]{GKP94}}]\label{def:continuant}
Let $Q$ be  a sequence  of elements $(q_1,q_2,\ldots,q_n)$ of a ring $R$.
We associate with  $Q$ an element $[Q]$ of $R$ via the following recurrence formula
\begin{align*}
[\,]&=1, [q_1]=q_1, [q_1,q_2]=q_1q_2+1, \text{~and}\\
[q_1,q_2,\ldots,q_n]&=[q_1,\ldots,q_{n-1}]q_n+[q_1,\ldots,q_{n-2}] \text{~ if~} n\ge 3.
\end{align*}
The value $[Q]$ is called the \emph{continuant} of the sequence
$Q$.
\end{definition}

A sequence $(q_1,q_2,\ldots,q_n)$ of quotients given by the Euclidean algorithm on $\tau_1$ and $\tau_2$, with $\tau_1$ and $\tau_2$ in $R$, is called a
\emph{continuant representation} of $(\tau_1,\tau_2)$ as we have the equalities
$\tau_1=[q_1,q_2,\ldots,q_n]h$ and
$\tau_2=[q_2,\ldots,q_n]h$
unless $\tau_2=0$.  If $\tau_2\ne 0$, then $h$ is  a gcd of $(\tau_1,\tau_2)$, else $h=\tau_1$;
in other words
$R\tau_1+R\tau_2=Rh$, where $R\tau$ denotes the left ideal generated by $\tau$.

Continuants have prominently featured in the literature.  For commutative rings many continuant properties are  given in \cite[Sec.~6.7]{GKP94}, while for non-commutative rings a careful study is presented in \cite{Wed1913}. For applications of continuants to representations of integers by quadratic forms, see \cite{Bri72,Bar 72,Wil1995,Wil1994,HMW90}. In all these papers, continuants have featured as numerators (and denominators) of continued fractions. For instance, the continuant $[q_1,q_2,q_3]$ equals the numerator of the continued fraction $q_1 + \cfrac{1}{q_2+\cfrac{1}{q_3}}$, while the continuant $[q_2,q_3]$ equals its denominator.

Let $p$ be a prime  number of the form $4k+1$.
Smith's approach \cite{CELV99} relies on the existence of a palindromic sequence
$Q=(q_1,\ldots , q_s,q_s,\ldots ,q_1)$ of even length such that $p=[Q]$.
He then derives a solution $z_0$ for $z^2+1\equiv0 \pmod{p}$ with $2\le z_0\le p/2$, namely
$[q_2,\ldots , q_s,q_s,\ldots ,q_1]$. On the other hand,
from $z_0$ one can retrieve the palindromic sequence by applying the
Euclidean algorithm to $p$ and $z_0$, and then $p=x^2+y^2$ where $x=[q_1,\ldots,q_s]$ and
$y=[q_1,\ldots, q_{s-1}]$.

With regards to the question of finding square roots modulo a prime $p$, a deterministic algorithm can be found in \cite{Sch1985}. The paper \cite{Wag90} also discusses the topic.

Brillhart's optimisation \cite{Bri72} on Smith's construction took full advantage
of the palindromic structure of the sequence \[(q_1,\dots,q_{s-1},q_s,q_s,q_{s-1},\dots,q_1)\]
given by the Euclidean algorithm on $p$ and $z_0$, a solution of $z^2+1\equiv0\pmod{p}$. He noted that the Euclidean
algorithm gives the remainders
\begin{align*}
r_i&=[q_{i+2},\dots,q_{s-1},q_s,q_s,q_{s-1},\dots,q_1] \text{ ($i=1,\ldots,s-2$),}\\
r_i&=[q_{2s-1-i},\dots,q_1] \text{ ($i=s-1,\ldots,2s-2$),}\\
r_{2s-1}&=[\,], \text{and}\\
r_{2s}&=0.
\end{align*}
 so, in virtue of Smith's construction, rather than computing the whole sequence we
need to obtain
\[
\begin{cases}
x&=r_{s-1}=[q_s,q_{s-1},\dots,q_1]\\
y&=r_{s}=[q_{s-1},\dots,q_1].
\end{cases}
\]
In this case, we have $y<x<\sqrt{p}$, Brillhart's stopping criterium.

\subsection{Previous extensions of Fermat's two-square theorem to other rings}

The question of extending Fermat's two-square theorem to other rings has been extensively considered in the literature; see, for instance, \cite{Niv1940,Hsi1973/74,ElM2007,Eli2006,CLR1980,Lea1967,Mor1932,Har1968}.

Quadratic fields have naturally received much attention. Niven \cite{Niv1940} considered imaginary quadratic fields and studied the problem of expressing an integer $a+2b\sqrt{-h}$ as a sum of two squares of integers in the field. Alternative proofs for the case of Gaussian integers (i.e.~ $h=1$) appeared in \cite{Lea65,Mor67,Wil73}. The number of representations of non-zero Gaussian integers as sums of two Gausssian integers was obtained by Pall \cite{Pal51}, and later by Williams \cite{Wil67,Wil71}. Elia \cite{Eli2006} proved that a totally positive integer $m$ in $\mathbb{Q}(\frac{1+\sqrt{5}}{2})$ is a sum of two squares iff in the prime decomposition of $m$ each of its prime factors of field norm congruent to $11,19$ modulo 20 occurs with an even exponent. An integer in a quadratic field is called {\it totally positive} if it and its conjugate are positive. Later, Elia and Monico \cite{ElM2007} obtained a similar result for  totally positive integers  in $\mathbb{Q}(\sqrt{2})$. Deutsch \cite[Thm.~6]{Deu02} also considered the field $\mathbb{Q}(\frac{1+\sqrt{5}}{2})$ and proved that a prime with $-1$ as a quadratic residue has a representation, up to multiplication by a unit, as a sum of two squares of integers in $\mathbb{Q}(\frac{1+\sqrt{5}}{2})$. 

Many results in this area rely on theorems about binary quadratic forms. For instance, Niven's proof of the aforementioned result heavily depends on a theorem by Mordell \cite{Mor1932}. In  \cite{Mor1932} Mordell  gave necessary and sufficient conditions for a positive binary quadratic form $ax^2+2hxy+by^2$ with integral coefficients to be representable as a sum of the squares of two linear forms $a_1x+b_1y$ and $a_2x+b_2y$ with integral coefficients. The number of representations of $ax^2+2hxy+by^2$ in the aforementioned manner was given in \cite{Pal51,Wil72}.  Mordell's result was subsequently extended by Hardy \cite{Har1968} to forms with Gaussian integers as coefficients. See also \cite{LeoWil73,Wil71a,LeoWil74}.

Polynomial rings have also attracted much attention. Hsia \cite{Hsi1973/74} studied the representation of cyclotomic polynomials as the sum of two squares in $K[X]$, where $K$ is an algebraic field.  Leahey \cite{Lea1967} proved a theorem in the same vein as Fermat's two-square theorem for polynomials in $\mathbb{F}[X]$, where $\mathbb{F}$ is a finite field of characteristic different from 2 and $-1$ is a non-square in  $\mathbb{F}$. Leahey's theorem reads as follows:

\begin{theorem}[{\cite{Lea1967}}]
Let $m\in \mathbb{F}[X]$ be a monic polynomial, then  any associate of $m$ is a sum of two squares iff in the prime decomposition of $m$ each of its prime factors of odd degree occurs with an even exponent.
\end{theorem}

Perhaps one of the most important extensions of Fermat's two-square theorem was given by Choi, Lam, Reznick and Rosenberg \cite{CLR1980}. In \cite{CLR1980} Choi et al. proved the following theorem.

\begin{theorem}[{\cite[Thm.~2.5]{CLR1980}}]Let $R$ be an integral domain, $\mathbb{F}_R$ its field of fractions, $-h$ a non-square in $\mathbb{F}_R$ and $R[\sqrt{-h}]$ the smallest ring containing $R$ and $\sqrt{-h}$. 
 
 If both $R$ and $R[\sqrt{-h}]$ are unique factorisation domains, then the following assertions hold.
 
 \begin{enumerate}
\item Any element $m\in R$ representable by the form $x'^2+hy'^2$ with $x',y'\in \mathbb{F}_R$ is also representable by the form $x^2+hy^2$ with $x,y\in R$.

\item Any element $m\in R$ representable by the form $x^2+hy^2$ can be factored into $p_1^2\cdots p_k^2q_1\cdots q_l$, where $p_i,q_j$ are irreducible elements in $R$ and $q_j$ is representable by $x^2+hy^2$ for all $j$.
\item An associate of a non-null prime element $p\in R$ is representable by $x^2+hy^2$ iff $-h$ is a square in $\mathbb{F}_{R/Rp}$, where $\mathbb{F}_{R/Rp}$  denotes the field of fractions of the quotient ring $R/Rp$.
\end{enumerate}
\end{theorem}

\subsection{Our work}

In this paper we study  \emph{proper} representations $x^2+y^2$ (that is, with $x$ and $y$ coprime)
in
some Euclidean rings via continuants. Specifically, we concentrate on the following problems. Below a unit in the ring is denoted by $u$.

\begin{problem}[From $x^2+y^2$ to $z^2+1$]If $m=u(x^2+y^2)$ and $x,y$ are coprime, can we find $z$ such that $z^2+1$ is a multiple of $m$ using continuants?
\end{problem}

\begin{problem}[From $z^2+1$ to $x^2+y^2$] If $m$ divides $z^2+1$, can we find $x,y$ such that $m=u(x^2+y^2)$ using continuants?
\end{problem}

As far as we know, this paper presents for the first time the application of continuants to representations in Euclidean rings other than the integers. Specifically, we present the following new deterministic algorithms for the form $Q(x,y)=x^2+y^2$. 
\begin{enumerate}
\item Algorithm \ref{DivAlg_FromQ(x,y)ToQ(z,1)}: for every $m$ in a commutative Euclidean ring, it finds a solution $z_0$ of $Q(z,1)\equiv0\pmod{m}$, given a representation $uQ(x,y)$ of $m$.
\item Algorithm \ref{PolyDivAlg}: for every polynomial $m\in \mathbb{F}[X]$, where $\mathbb{F}$ is a field of odd characteristic, it finds a proper representation $uQ(x,y)$ of $m$, given a solution $z_0$ of $Q(z,1)\equiv0\pmod{m}$. 
\end{enumerate}

As an application of continuants, we provide a new constructive proof of the four-square theorem (Section \ref{4sq}). Many proofs of this theorem can be found in the literature; see, for instance, \cite[Sec.~20.5, 20.9, 20.12]{HW08} and \cite{Hir87, Sma82,AndEkhZei93,SpeWil00,AlaAlaLem07}. 

Furthermore, we use continuants to prove a number of other results about quaternary forms which represent all integers (Section \ref{otherform}).

From the outset we emphasise that Smith's approach heavily depends on the existence of a Euclidean-like division algorithm and that, if one tries to extend it to other Euclidean rings $R$, the uniqueness
of the continuant representation may be lost. The uniqueness of the
continuant representation boils down to the uniqueness of the quotients
and the remainders in the division algorithm. This uniqueness is achieved
only when $R$ is a field or $R = \F[X]$, the polynomial algebra over a field $\F$ (considering the degree as the Euclidean function) \cite{Jod67}. Note that in $\Z$ the uniqueness is guaranteed by requiring the remainder to be nonnegative.

The rest of the paper is structured as follows. In Section \ref{sec:Continuants} we  study
properties of continuants in arbitrary rings.
Section \ref{sec:EucRings} is devoted to studying proper representations $x^2+y^2$ in some Euclidean rings.
We examine later some representations $x\overline x+y\overline y$ using rings with
an anti-automorphism $x\mapsto \overline{x}$ (Sections~\ref{4sq} and \ref{otherform}).
\section{Continuants}
\label{sec:Continuants}

In this section we derive some properties of continuants from
Definition~\ref{def:continuant},
which we will refer to as continuant properties. Many of these properties are already known; see \cite[Sec.~ 6.7]{GKP94} and \cite{Wed1913}.

\begin{mylist}
\item \label{PPeuler}
The first property is the so-called ``Euler's rule'' \cite[pp.~72]{Dav08}: Given a sequence $Q$, compute all the products of subsequences of $Q$ obtained by removing
disjoint pairs of consecutive elements of $Q$. Then the continuant
$[Q]$  is given by the sum of all such products. The empty product is 1, as usual.
\begin{example}
Consider $Q=(q_1,q_2,q_3,q_4,q_5)$. Then the products of relevant subsequences are: $q_1q_2q_3q_4q_5$, $q_3q_4q_5$, $q_1q_4q_5$, $q_1q_2q_5$, $q_1q_2q_3$, $q_5$, $q_3$, and $q_1$. Thus, the continuant is\begin{align*}[Q]=&{} q_1q_2q_3q_4q_5+q_3q_4q_5+q_1q_4q_5+q_1q_2q_5+q_1q_2q_3+\\
&+q_5+q_3+q_1.\end{align*}
\end{example}
\item \label{PPzigzag}
If in a ring $R$ we find a unit $\tau$ commuting with all $q_i$'s, then
\[
[\tau ^{-1}q_1,\tau q_2,\ldots, \tau ^{(-1)^k}q_k,\ldots, \tau  ^{(-1)^n}q_n]=\begin{cases}
[q_1,\ldots,q_n] &\text {if $n$ even}\\
\tau ^{-1}[q_1,\ldots,q_n] &\text {if $n$ odd}
\end{cases}
\]

\item\label{PPcutting}
 $[q_1,\ldots, q_n]=[q_1,\ldots, q_{i-1}][q_{i+2},\ldots, q_n]+[q_1,\ldots, q_{i}][q_{i+1},\ldots, q_n]$. To obtain this equality,  it suffices to
  divide  the products of subsequences of $Q=(q_1,q_2,\ldots,q_n)$
obtained by removing disjoint pairs of consecutive elements of $Q$ into two groups,
depending on whether the pair $q_iq_{i+1}$ ($1\le i<n$) has been removed or not.

\item \label{PPbackto0}
From the previous points it follows
\begin{multline*}
[-q_h,-q_{h-1},\ldots,-q_1,0,q_1,q_2,\ldots, q_n]=\\
\begin{cases}
    [q_{h+2},q_{h+3},\ldots, q_n]&  \text{~for~} 0\le h\le n-2\\
    1& \text{~if~} h=n-1\\
    0& \text{~if~} h=n\end{cases}
\end{multline*}

\item\label{PPcoprime}
  $[q_1,\ldots, q_n]$ and $[q_1,\ldots, q_{n-1}]$ are coprime.
\end{mylist}

From Properties P--\ref{PPzigzag}
and P--\ref{PPbackto0}, we can derive more identities, for instance, the following.

\begin{mylist}

\item \label{PPBezout}
$[-q_{n-1},\ldots, -q_1,0][q_1,\ldots,q_n]+[-q_{n-1},\ldots, -q_1][q_2,\ldots,q_n]=1$.

\item \label{PPWed} Property P--\ref{PPBezout} is equivalent to
\[[-q_{n-1},\ldots, -q_2][q_1,\ldots,q_n]+[-q_{n-1},\ldots, -q_1][q_2,\ldots,q_n]=1,\]
which is in turn equivalent to
\[[q_{n-1},\ldots, q_2][q_1,\ldots,q_n]-[q_{n-1},\ldots, q_1][q_2,\ldots,q_n]=(-1)^n\]
\end{mylist}
This last property first appeared in Theorem 3 of \cite{Wed1913}, where other variants were also presented.

If the ring $R$ is commutative, then we have some additional properties.
\begin{mylist}
\item\label{PPreverse}
$[q_1,q_2,\ldots,q_n]=[q_n,\ldots,q_2,q_1]$.

\item\label{PPdeterminant}
 The continuant $[q_1,\ldots, q_n]$
is the determinant of the tridiagonal $n\times n$ matrix $A=(a_{ij})$ with $a_{i,i}=q_i$
for $1\le i\le n$, $a_{i,i+1}=1$ and $a_{i+1,i}=-1$ for $1\le i<n$.
\end{mylist}

The following identity, due to Lewis Carroll (alias Charles Lutwidge Dodgson), 
plays an important role in our study of continuants.

\begin{lemma}[Lewis-Carroll's identity, {\cite{Dod1866}}] \label{lemm:LC_Ident}
Let $C$ be an $n\times n$ matrix in a commutative ring. Let $C_{i_1,\ldots, i_s;j_1,\ldots, j_s}$
denote the matrix obtained from $C$ by omitting the rows $i_1,\ldots, i_s$ and the columns
$j_1,\ldots, j_s$. Then
\[
\det(C)\det(C_{i,j;i,j})=\det(C_{i;i})\det(C_{j;j})-\det(C_{i;j})\det(C_{j;i})
\]
where  the determinant of the
0$\times$0 matrix is 1 for convenience.
\end{lemma}

The use of Lewis-Carroll's identity and Property P--\ref{PPdeterminant}
provides more properties.
\begin{mylist}
\item\label{PPlewiscarroll}
 $[q_1,q_2,\ldots, q_{n}][q_2,\ldots , q_{n-1}]= [q_1,\ldots,q_{n-1}][q_2,\ldots , q_n]+(-1)^{n}$
(when $n\ge 2$).
\item\label{PPpalindro}
 In the case of even $n$ with $q_i=q_{n+1-i}$ for $1\le i\le n$ (i.e~
if the sequence is \emph{palindromic}), we have
    \begin{align*}
    &[q_1,\ldots,q_{n/2},q_{n/2},\ldots, q_2]^2+1=\\
 &[q_1,\ldots,q_{n/2},q_{n/2},\ldots, q_1]
[q_2,\ldots,q_{n/2},q_{n/2},\ldots, q_2]=\\
&([q_1,\ldots,q_{n/2}]^2+[q_1,\ldots, q_{n/2-1}]^2)
([q_2,\ldots,q_{n/2}]^2+[q_2,\ldots, q_{n/2-1}]^2)
 \end{align*}
\end{mylist}

Note that Property P-\ref{PPlewiscarroll} also follows from Properties P--\ref{PPWed} and P--\ref{PPreverse}. More properties and proof techniques
for the commutative case are given in \cite[Sec.~6.7]{GKP94}
\subsection{Quasi-palindromic sequences}
Here again the rings are not necessarily commutative.
\begin{definition}\label{antimorph}
An \emph{anti-automorphism} of a ring $R$
is an involution $\tau\mapsto  \overline \tau$ such that
$\overline{\tau+\sigma}=\overline{\tau}+\overline{\sigma}$ and
$\overline{\tau\sigma}=\overline \sigma\,\overline \tau$ for all
elements $\tau$, $\sigma$ of $R$.
\end{definition}

\begin{definition}\label{quasipal}
Let $R$ be a ring endowed with an anti-auto\-mor\-phism $\tau\mapsto \overline \tau$.
A \emph{quasi-palin\-dro\-mic} sequence of length $n$
satisfies $q_i=\overline{ q_{n+1-i}}$ for $1\le i\le n$;  in particular,
if $n$ is odd the element $q_{(n+1)/2}$ satisfies
$q_{(n+1)/2}=\overline{q_{(n+1)/2}}$.
\end{definition}
We have  an obvious relation.
\begin{mylist}
\item\label{PPantiauto}
$[\overline {q_n},\ldots, \overline{q_1}]=\overline{[q_1,\ldots q_n]}$
\end{mylist}
and counterparts of  Properties P--\ref{PPlewiscarroll} and P--\ref{PPpalindro}.

\begin{lemma}[Noncommutative Lewis-Carroll-like identity]
\label{lemm:LC_IdentNonComm}
 Let $\tau \mapsto \overline{\tau}$ be an anti-auto\-mor\-phism in a ring $R$,
which also satisfies the conditions
\begin{equation}\label{special}
\begin{cases} \tau\overline{\tau}=\overline{\tau}\tau, \;\text{and}\\
 \text{if~}\overline{\tau}=\tau \text{~then~}\tau\text{~belongs to the centre of~} R.
 \end{cases}
\end{equation}
Let  $(q_1,\ldots q_n)$ be a quasi-palin\-dro\-mic sequence of length $n\ge 2$ in $R$.
The following relation holds

$\begin{array}{l@{{}={}}l}[q_1,\ldots, q_n][q_2,\ldots, q_{n-1}]&
[q_2,\ldots, q_n][q_1,\ldots, q_{n-1}]+(-1)^n\\
&[q_1,\ldots, q_{n-1}][q_2,\ldots, q_n]+(-1)^n.
\end{array}$
\end{lemma}
\begin{proof} We  proceed by induction on $n$.
Our basic cases are $n=2,3$. The result is clearly true for $n=2$.

For $n=3$, since $q_2$ is in the centre of $R$ and $q_1$ commutes with $q_3$, from
\begin{align*}
[q_1,q_2][q_2,q_3]-1&=(q_1q_2+1)(q_2q_3+1)-1\\
&=q_1q_2q_2q_3+q_1q_2+q_2q_3\\
&=q_1q_2q_2q_3+q_1q_2+q_2q_3
\end{align*}
we obtain $q_1q_2q_2q_3+q_1q_2+q_2q_3=[q_2,q_3][q_1,q_2]-1=[q_1,q_2,q_3]q_2$.

For larger $n$, write $E=[q_2,\ldots,q_{n-1}]$ and
$F=[q_3,\ldots,q_{n-2}]$. We want to prove that 
$\begin{array}{l@{{}={}}l}[q_1,\ldots, q_n]E
&[q_1,\ldots, q_{n-1}][q_2,\ldots, q_n]+(-1)^n.
\end{array}$

Observe that $E$ and $F$ belong to the centre of $R$, and
that the following results come from the definition of continuant and Property P--\ref{PPcutting}.

On one hand, we have that  
\begin{align*}
[q_1,\ldots,q_{n-1}][q_2,\ldots,q_n]&=
(q_1 E + [q_3,\ldots, q_{n-1}])(E q_n + [q_2,\ldots, q_{n-2}])\\
&=q_1E^2q_n+q_1E[q_2,\ldots,q_{n-2}]+
[q_3,\ldots, q_{n-1}]E q_n\\
&\quad+[q_3,\ldots, q_{n-1}][q_2,\ldots, q_{n-2}].
\end{align*}
On the other hand, we have that
\begin{align*}
[q_1,q_2,\ldots,q_{n-1}, q_n]E&=
(q_1[q_2,\ldots,q_{n-1}, q_n]+[q_3,\ldots, q_n])E\\
&=\left(q_1( E q_n +[q_2,\ldots,q_{n-2}])+ [q_3,\ldots, q_{n-1}]q_n+F\right)E\\
&=q_1E q_nE+q_1[q_2,\ldots,q_{n-2}]E+[q_3,\ldots, q_{n-1}]q_nE\\
&\quad +FE.
\end{align*}

First note that $[q_2,\ldots,q_{n}][q_1,\ldots,q_{n-1}]=[q_1,\ldots,q_{n-1}][q_2,\ldots,q_n]$
because of the equality  $ [q_1, \ldots,q_{n-1}] =  \overline {[q_2,\ldots ,q_n]}$.

Since $E=\overline{E}$, $E$ commutes with the whole $R$ and we have
\begin{align*}
q_1 E^2 q_n&=q_1 E q_n E\\
 q_1 E [q_2,\ldots, q_{n-2}]&=q_1[q_2,\ldots,q_{n-2}]E,\text{ and}\\
[q_3,\ldots, q_{n-1}]Eq_n&=[q_3,\ldots, q_{n-1}]q_n E.
\end{align*}

It only remains to check
\begin{align*}
EF&=[q_2,\ldots, q_{n-2}][q_3,\ldots,q_{n-1}]+(-1)^n\\
&=[q_3,\ldots, q_{n-1}][q_2,\ldots, q_{n-2}]+(-1)^n,
\end{align*}
but these equalities follows from the inductive hypothesis.
\end{proof}

\begin{remark}\label{revanti}
For a quasi-palindromic sequence $Q$ of length $n\ge 3$, we have
\begin{align*}
[q_1,q_2,\ldots,q_{n-1}]&=q_1[q_2,\ldots,q_{n-1}]+[q_3,\ldots,q_{n-1}]\\
&=q_1[q_2,\ldots,q_{n-1}]+\overline{[q_2,\ldots,q_{n-2}]}
\end{align*}
\end{remark}

\section{Proper representations in Euclidean rings}
\label{sec:EucRings}

As said before, if one tries to extend Smith's approach to other Euclidean rings $R$, the uniqueness of the continuant representation may be lost. Given two elements $m,z\in R$, the uniqueness of the
continuant representation of $(m,z)$ is necessary to recover representations $m=x\overline{x}+y\overline{y}$  from a multiple $z\overline{z}+1$ of $m$.

\subsection{Euclidean rings not necessarily commutative}
We first use continuants to obtain a multiple $z\overline{z}+1$  of  an element $m$ of
the form $x\overline{x}+y\overline{y}$, with $x,y$ satisfying $Rx+Ry=R$ and
$\tau \mapsto \overline{\tau}$ an anti-automorphism in the ring under consideration.

\begin{theorem}\label{theo:SqSumNonComm} Let $R$ be a Euclidean ring, and let
$\tau\mapsto \overline{\tau}$ be an anti-automor\-phism of $R$ satisfying the conditions (\ref{special}) of Lemma ~\ref{lemm:LC_IdentNonComm}.
If $m \in R$
admits a \emph{proper} representation $m=x\overline{x}+y\overline{y}$ (that is, with
$Rx+Ry=R$),
then the equation $z\overline{z}+1\in Rm$ admits  solutions.

Furthermore, one   of these solutions  is equal to
$[\overline{q_s},\ldots,\overline{q_1},q_1,\ldots,q_{s-1}]$, where
$(q_1,q_{2},\dots,q_{s})$ is a sequence provided by the Euclidean algorithm on
$x$ and $y$.
\end{theorem}
\begin{proof} Let $\nm$ denote the Euclidean function of $R$ and let $(x,y)$ (with $\nm(x)\ge \nm(y)$) be a proper representation of $m$.

If $y=0$ then $x$ is a unit, so $m$ must be a unit and the ideal  $Rm$ is the
whole ring $R$. Otherwise, the  Euclidean algorithm on $x$ and $y$ gives a unit $u$
and a sequence $(q_1,q_2,\ldots,q_s)$ such that  $x=[q_1,q_2,\ldots,q_s]u$ and $y=[q_2,\dots,q_s]u$.
Then
\begin{align*}
x\overline{x}&= [q_1,\ldots,q_s]u\overline{u}
[\overline{q_s},\ldots,\overline{q_1}],
\text{~using Property P--\ref{PPantiauto}}\\
x\overline{x}&= [\overline{q_s},\ldots,\overline{q_1}]
[q_1,\ldots,q_s]u\overline{u}, \text{~since $u\overline{u}$ belongs to the centre of $R$ }\\
y\overline{y}&=[\overline{q_s},\ldots,\overline{q_2}][q_2,\ldots,q_s]u\overline{u}\\
m&=x\overline{x}+y\overline{y}=
[\overline{q_s},\ldots,\overline{q_1},q_1,\ldots,q_s]u\overline{u},
\text{~by Property P--\ref{PPcutting}}
\end{align*}
Let $z=[\overline{q_s},\ldots,\overline{q_1},q_1,\ldots,q_{s-1}]$. Then applying Lemma
\ref{lemm:LC_IdentNonComm} we obtain
\begin{multline*}
z\overline z +1=(u\overline{u})^{-1}m
[\overline{q_{s-1}},\ldots,\overline{q_1},q_1,\ldots,q_{s-1}]\\
=(u\overline{u})^{-1}[\overline{q_{s-1}},\ldots,\overline{q_1},q_1,\ldots,q_{s-1}]m\\
\text{since $m$ is in the center of $R$}
\end{multline*}
That is, $z$ satisfies $z\overline{z}+1\in Rm$, which completes
the proof of the theorem.
\end{proof}

\subsection{Commutative rings: from   $x^2+y^2$ to $z^2+1$}
 In this subsection we deal with the problem of going from a representation $x^2+y^2$ of an associate of an element $m$ to a multiple $z^2+1$ of $m$. We begin with a very general remark valid in every commutative ring.

\begin{corollary}\label{direct}
In a commutative ring $R$, if $Rx+Ry=R$ then there exists some $z\in R$ such that $x ^2+y^2$
divides $z^2+1$.

If $R$ is Euclidean, we can explicitly find $z$  and the quotient $(z^2+1)/(x^2+y^2)$ with continuants.
\end{corollary}

This relation can be interpreted using Lewis-Carroll's identity.
The determinant of the tridiagonal matrix $A$ associated with the palindromic sequence
$(q_s,\ldots,q_1,\allowbreak q_1,\ldots,q_s)$
(see property P--\ref{PPdeterminant} of continuants)
is $x^2+y^2$ with $x=[q_1,\ldots, q_s]$ and $y=[q_2,\ldots, q_s]$ if  $s\ge1$.

Moreover,
$(x^2+y^2)([q_1,\ldots q_{s-1}]^2+[q_2,\ldots q_{s-1}]^2)= z^2+1$,
where  $z$ is the determinant of matrix formed by the $2s-1$ first rows and columns of $A$ (see properties P--\ref{PPlewiscarroll} and P--\ref{PPreverse}). These remarks can readily be converted into a deterministic algorithm; See Algorithm \ref{DivAlg_FromQ(x,y)ToQ(z,1)}.

\begin{algorithm}[ht]
\SetKwData{Left}{left}\SetKwData{This}{this}\SetKwData{Up}{up}
\SetKwFunction{Union}{Union}\SetKwFunction{FindCompress}{FindCompress}
\SetKwInOut{Input}{input}\SetKwInOut{Output}{output}
\SetKw{Init}{initialisation:}
\SetKw{Ass}{assumptions:}
\Input{A commutative Euclidean ring $R$.\\An element $m\in R$.\\A proper representation $uQ(x,y)$ of $m$, where $Q(x,y)=x^2+y^2$.}
\Output{A solution $z_0$ of $Q(z,1)\equiv0\pmod{m}$ with $\nm(1)\le \nm(z_0)$.}
\tcc{Apply the Euclidean algorithm to $x$ and $y$ and obtain a sequence $(q_1,\ldots,q_s)$ of quotients.}
$s\leftarrow 0$\;
$m_0 \leftarrow m$\;
$r_0 \leftarrow z$\;
\Repeat{$r_{s}= 0$}
{
$s\leftarrow s+1$\;
$m_{s}\leftarrow r_{s-1}$\;
find $q_{s}, r_{s}\in R$ such that $m_{s-1} = q_{s} m_{s} +r_{s}$ with
$\nm(r_{s})<\nm(m_{s})$\;
}
$z_0 \leftarrow [q_s,q_{s-1},\ldots,q_1,q_1,q_2,\ldots,q_{s-1}]$\;
\Return $z_0$
\caption{Deterministic algorithm for constructing a solution $z_0$ of $Q(z,1)\equiv0\pmod{m}$, given a representation $uQ(x,y)$ of an element $m$.}
\label{DivAlg_FromQ(x,y)ToQ(z,1)}
\end{algorithm}

\subsection{Commutative rings: from $z^2+1$ to $x^2+y^2$}
Here we deal with the problem of going from a solution $z_0$ of $z^2+1\equiv0\pmod{m}$ to a representation $x^2+y^2$ of an associate of $m$.

A natural question goes as follows: if $m$ divides $z^2+1$, does there exist $x,y$ such that
$m=x^2+y^2$? We now give examples showing
that no simple answer is to be expected.

In general, we cannot construct a representation $x^2+y^2$ of an element $m$
from a solution of  $z^2+ 1\equiv 0 \pmod m$. As an illustration, consider
 the Euclidean domain $\F_2[X]$  of polynomials over the field $\F_2$, where $z^2+1$
is a multiple of $m=z+1$ for any polynomial $z$, square or not. Recall that in $\F_2[X]$
the squares, and therefore the sums of squares, are exactly the even polynomials
(i.e.~the coefficient of $X^t$ is null if $t$ is odd). Thus, the converse of Corollary
\ref{direct} is false in $\F_2[X]$. Other examples are  the ring $\Z[i]$ of Gaussian integers and its quotients
by an even integer, since the squares and the sum of squares have an even
imaginary part. Thus, no Gaussian integer with an odd imaginary part is
a sum of squares, although it obviously divides $0=i ^2+1$; see \cite[Sec.~3]{Niv1940}.

However, there are cases where the answer is positive. Propositions \ref{Prop:2Inv}-\ref{applinv} discuss some of these cases.
\begin{proposition}
\label{Prop:2Inv}
Let $R$ be a commutative ring.
If 2 is invertible and $-1$ is a square, say $1+k^2=0$, then
$x=\left(\frac{x+1}2\right)^2+\left(\frac{x-1}{2k}\right)^2$.
\end{proposition}

Variants of Proposition \ref{Prop:2Inv} have appeared previously in the literature. For instance, a variant can be found  in \cite[p.~817]{Lea1967} in the context of finite fields.

\begin{proposition}\label{inverse}
Let $R=\F[X]$ be the ring of polynomials over a field  $\F$  with characteristic different from 2 and let  $-1$ be a non-square in $\F$.

If $m$ divides $z^2+t^2$ with $z,t$ coprime, then $m$ is an associate of some $x^2+y^2$ with $x,y$ coprime.
\end{proposition}
\begin{proof}
We introduce the extension $\G$ of $\F$
by a square root $\omega$ of $-1$. The ring $\G[X]$ is principal and $z^2+t^2$ factorises as
$(z-\omega t)(z+\omega t)$. The two factors are coprime, since their sum and difference
are respectively $2z$ and $2\omega t$, and $2$ and $\omega$ are units. Introduce
$\gcd(m,z+\omega t)=x+\omega y$, then $x-\omega y$ is a gcd of $m$ and $z-\omega t$ owing to the
natural automorphism of $\G$. The polynomials $x-\omega y$ and $x+\omega y$ are coprime
and both divide $m$. Thus, $m$ is divisible by $(x-\omega y)(x+\omega y)=x^2+y^2$. On the other hand,  $m$ divides
$(z-\omega t)(z+\omega t)$. Consequently,  $(x-\omega y)(x+\omega y)$ is an associate of $m$. Since  $x-\omega y$ and $x+\omega y$ are coprime, we have $x$, $y$ are coprime.
\end{proof}
On one hand, Corollary \ref{direct} and Proposition \ref{inverse} somehow generalise the main theorem of \cite{Lea1967}. On the other hand, in the case of $m$ being prime, Proposition \ref{inverse} is embedded in Theorem 2.5 of \cite{CLR1980}.

\begin{proposition}\label{applinv}
Let $m$ be a non-unit of $\F[X]$ and a divisor of $z^2+1$ for some $z\in \F[X]$ with $\deg(z)<\deg(m)$.

If $\F$ is a field of characteristic different from 2, where  $-1$ is a non-square, then
continuants provide a method for representing $m$ as a sum of squares.

Specifically, the Euclidean algorithm on $m$ and $z$ gives the unit $u$ and the sequence  $(uq_s,u^{-1}q_{s-1},\ldots,u^{(-1)^{s-1}}q_1,u^{(-1)^{s}}q_1,\ldots, u^{-1}q_s)$ such that $x=[q_1,\ldots, q_s]$ and $y=[q_2,\ldots,q_{s}]$.
\end{proposition}
\begin{proof}
Having a divisor $m$ of $z^2+1$, we already know from Proposition~\ref{inverse}
that the degree of $m$ is even. We may assume that
$\mathrm{degree}(z)<\mathrm{degree}(m)$ as we may divide $z$ by $m$.

From Proposition~\ref{inverse} we also know that, for this given $z$,  $m/u=x^2+y^2$ for some coprime $x,y$. Consequently, the Euclidean algorithm on these $x$ and $y$ will give the unit 1 and the sequence $(q_1,\ldots, q_s)$ such that $x=[q_1,\ldots q_s]$, $y=[q_2,\ldots,q_s]$ and $m/u=[q_s,\ldots,q_1,q_1,\ldots,q_s]$. Theorem \ref{theo:SqSumNonComm} tells that, given these $x$ and $y$, the element $z$ has the form $[q_s,\ldots,q_1,q_1,\ldots,q_{s-1}]$, which, by Property P\ref{PPreverse}, is equivalent to $[q_{s-1},\ldots,q_1,q_1,\ldots,q_s]$. Note that the uniqueness of the continuant representation of $(x,y)$ has implicitly been invoked.

We may also assume $\mathrm{degree}(x)>\mathrm{degree}(y)$, otherwise, if $x=\lambda y +t$ with $\lambda$ a unit and  $t$ a polynomial of degree smaller than the degree of $x$ and $y$, then
$m= \left(((1+\lambda ^2)y+\lambda t)^2+t^2\right)\frac u{1+\lambda ^2}$. As a result, we consider
only the case where all $q_i$'s have degree at least 1 in the continuant representation of $(x,y)$.

We then apply the Euclidean algorithm to $m$ and $z$, and obtain,
by virtue of the uniqueness of the division in polynomials, a sequence whose last non-null remainder is $u$. Consequently, $m/u=x^2+y^2$ (see Property P--\ref{PPzigzag} of continuants).
\end{proof}
We illustrate this proposition through some examples. First take $m=2X^4-2X^3+3X^2-2X+1$, then $m$ divides $(2X^3+X)^2+1$. The Euclidean divisions give successively
\begin{align*}
2X^4-2X^3+3X^2-2X+1=&(2X^3+X)(X-1)+2X^2-X+1\\
2X^3+X=&(2X^2-X+1)(X+1/2)+(X/2-1/2)\\
2X^2-X+1=&(X/2-1/2)(4X+2)+2\\
X/2-1/2=&2(X/4-1/4).
\end{align*}
 Here we have $m/2=[2\cdot(X-1)/2,2^{-1}\cdot(2X+1),2\cdot(2X+1),2^{-1}\cdot(X-1)/2]$ with $u=2$, which gives $m/2=(X^2-X/2+1/2)^2+(X/2-1/2)^2=x^2+y^2$. Since $2$ is also a sum  of two squares, we obtain $m=(x+y)^2+(x-y)^2=X^4+(X^2-X+1)^2$.

We find other examples among the cyclotomic polynomials. The cyclotomic polynomial $\Phi_{4n}\in \Q[X]$ divides $X^{2n}+1$.
Thus, $\Phi_{4n}$ is, up to a constant,
a sum  of two squares; see, for instance, \cite{Pou1971}.  Since $\Phi_{4n}(0)=1$, the constant can be chosen equal to 1. For an odd prime $p$, it is  easy to check
\[\Phi_{4p}(X)=\sum_{k=0}^{p-1}(-1)^kX^{2k}=
\left(\sum_{k=0}^{(p-1)/2}(-1)^kX^{2k}\right)^2+\left(X\sum_{k=0}^{(p-3)/2}(-1)^kX^{2k}\right)^2\]

For the small composite odd number 15, the computation gives
\begin{align*}
\Phi_{60}(X)&=X^{16}+X^{14}-X^{10}-X^8-X^6+X^2+1\\
&=[X,X,X^3-X,-X,-X,X, X,-X,-X,X^3-X,X,X]\\
X^{15}&=[X,X^3-X,-X,-X,X,X,-X,-X,X^3-X,X,X]\\
x&=[X,-X,-X,X^3-X,X,X]\\
  &=X^8-X^4+1\\
y&=[-X,-X,X^3-X,X,X]\\
&=X^7 +X^5 -X^3-X\\
\Phi_{60}(X)&=x^2+y^2
\end{align*}

At this stage the following remark is important.

\begin{remark}
If a polynomial with integer coefficients is the sum of squares of two polynomials
with rational coefficients,  it is also the sum of squares of two polynomials with
integer coefficients
\end{remark}
For example, we see that $50X^2+14X+1=(5X+3/5)^2+(5X+4/5)^2$, but it is also $X^2+(7X+1)^2$.

This remark follows from Theorem 2.5 of \cite{CLR1980}. Other proofs can be found in \cite{GZ} and \cite[Sec.~5]{DaLeSc}.

\begin{remark}[Algorithmic considerations]\label{practical} For the cases covered in Proposition \ref{applinv}, given an element $m$ and a solution $z_0$ of $z^2+1\equiv0\pmod{m}$, we can recover a representation $x^2+y^2$ of an associate of $m$ via continuants and Brillhart's \cite{Bri72} optimisation. We divide $m$ by $z_0$ and stop when we first encounter a remainder $r_{s-1}$ with degree at most $\deg(m)/2$. This will be the $(s-1)$-th remainder, and the quotients so far obtained are $(uq_s,u^{-1}q_{s-1},\ldots,u^{{-1}^{(s-2)}}q_2)$.  In this context
\begin{align*}
x=\begin{cases}
r_{s-1}&\textrm{for odd $s$}\\
u^{-1}r_{s-1}&\textrm{for even $s$}\\
\end{cases}
\end{align*}
\begin{align*}
y=\begin{cases}
[uq_s,u^{-1}q_{s-1},\ldots,u^{(-1)^{s-2}}q_2]&\textrm{for odd $s$}\\
u^{-1}[uq_s,u^{-1}q_{s-1},\ldots,u^{(-1)^{s-2}}q_2]&\textrm{for even $s$}\\
\end{cases}
\end{align*}
This observation follows from dividing $m/u=[q_s,\ldots,q_1,q_1,\ldots,q_s]$ by $z_0=[q_{s-1},\ldots,q_1,q_1,\ldots,q_s]$ using continuant properties.
\end{remark}

Algorithm \ref{PolyDivAlg} implements Remark \ref{practical}.

\begin{algorithm}[ht]
\SetKwData{Left}{left}\SetKwData{This}{this}\SetKwData{Up}{up}
\SetKwInOut{Input}{input}\SetKwInOut{Output}{output}
\SetKw{Init}{initialisation:}
\SetKw{Ass}{assumptions:}
\Input{A field $\F$ with characteristic different from 2.\\The ring $R=\F[X]$  of polynomials over $\F$.\\A polynomial  $m$ with $\nm(1)< \nm(m)$.\\A solution $z_0$ of $Q(z,1)\equiv0\pmod{m}$ with $\nm(1)< \nm(z_0)<\nm(m_0)$.}
\Output{A unit $u$ and a proper representation $uQ(x,y)$ of $m$.}
\tcc{Divide $m$ by $z$ using the Euclidean algorithm until we find a remainder $r_{s-1}$ with degree at most $\deg(m)/2$.}
$s\leftarrow1$\;
$m_0\leftarrow m$\;
$r_0\leftarrow z$\;
\Repeat{$\deg(r_{s-1})\le \deg(m)/2$}
{
$s\leftarrow s+1$\;
$m_{s-1}\leftarrow r_{s-2}$\;
find $k_{s-1}, r_{s-1}\in R$ such that $m_{s-2} = k_{s-1} m_{s-1} +r_{s-1}$ with
$\nm(r_{s-1})<\nm(m_{s-1})$\;
}
\tcc{Here we have a sequence $(k_1,\ldots,k_{s-1})$ of quotients.}
$x_{temp}\leftarrow r_{s-1}$\;
$y_{temp}\leftarrow [k_{1},\ldots,k_{s-1}]$\;
\tcc{We obtain a unit $u$.}
\If{$s$ is odd}{Solve $m=u(x_{temp}^2+y_{temp}^2)$ for $u$}\Else{Solve $um=x_{temp}^2+y_{temp}^2$ for $u$}
\tcc{We obtain $(x,y)$ so that  $m=(x^2+y^2)u$.}
\lIf{$s$ is odd}{$x\leftarrow x_{temp}$}\lElse{$x\leftarrow u^{-1}x_{temp}$}\;
\lIf{$s$ is odd}{$y\leftarrow y_{temp}$}\lElse{$y\leftarrow u^{-1}y_{temp}$}\;
\Return $(x,y,u)$
\caption{Deterministic algorithm for constructing a proper representation $uQ(x,y)=u(x^2+y^2)$ of an element $m$}
\label{PolyDivAlg}
\end{algorithm}

\section{Four-square theorem}\label{4sq}

The four-square theorem has been proved in a number of ways.  Hardy and Wright \cite[Sec.~20.5, 20.9, 20.12]{HW08} present three proofs: one based on the ``method of descent'', one based on quaternions, and one based on elliptic functions. We are aware of five other proofs \cite{Hir87, Sma82,AndEkhZei93,SpeWil00,AlaAlaLem07}. The proofs in \cite{Hir87,AndEkhZei93,AlaAlaLem07} are based on the triple-product identity, the paper \cite{SpeWil00} gives an arithmetic proof based on the number of representations of a positive number as the sum of two squares, and the proof in \cite{Sma82} is based on factorisations of $2\times 2$ matrices over the ring $\mathbb{Z}[i]$ of Gaussian integers.  
    
In this section, we provide a new constructive proof of the four-square theorem. Our proof is based on continuants over $\mathbb{Z}[i]$. 

We start by stating the following formula, which was already known to Euler, see \cite[pp.~277]{Dic71}.
\begin{lemma}[Product formula]\label{prodform}
Let $R$ be a commutative ring endowed with an anti-automorphism. Let
$x,y,z,u$ be elements of $R$.
Then
\[
(x\overline x+y\overline y)(z\overline z+u\overline u)=
(xz-y \overline u)(\overline{xz-y \overline u})+(xu+y\overline z)(\overline{xu+y\overline z})
\]
\end{lemma}
\begin{proof}
This can be seen by looking at the determinants in the equality
\[
\begin{barray}{2} x & y\\-\overline y & \overline x\end{barray}
\begin{barray}{2} z & u\\-\overline u & \overline z\end{barray}
=
\begin{barray}{2} xz-y\overline u  & xu+y\overline z\\
-\overline{ xu+y\overline z}  & \overline  {xz-y\overline u}\end{barray}.
\]\end{proof}

We use Lemma~\ref{prodform}
for the case of $R$ being
the ring of Gaussian integers, with its conjugation. This product formula allows to reduce the proof of the four-square theorem
to the case of primes.

We recall that each prime $p$ is either of the form $z\overline z$ or divides $z\overline z +1$, for some $z\in \mathbb{Z}[i]$ \cite[Prop.~4.18]{Cox89}. If $p=z\overline z$ then $p$ is trivially a sum of four squares. Assume the equation $z\overline z +1 \equiv 0\pmod p$ admits a solution $z_0$ over $\mathbb{Z}[i]$. Given this solution $z_0$, we prove the four-square theorem by constructing a representation of $p$ as $x\overline x+y\overline y$, with $x,y\in \mathbb{Z}[i]$. 

By reducing $z_0$ modulo $p$, we may assume $|z_0|\le p/\sqrt2$, and thus
$z_0\overline z_0+1< p^2$ (if $p=2$, a
parity argument shows the inequality remains valid). Here $|z_0|$ denotes the {\it complex norm} of $z_0$.

Let $p_0:=p$. Then, we produce a succession of $s$ equalities
$p_{i}p_{i+1}=z_i\overline{z_i}+1$ and $z_i=q_{i+1}p_{i+1}+z_{i+1}$,
where the sequence of positive integers $p=p_0,p_1,\ldots, p_s=1$ is decreasing.
At the end, we have $p_{s-1}p_s=z_{s-1}\overline{z_{s-1}}+1$ and $q_s=z_{s-1}$.

We now build a continuant representation of $q$. The equation $p_{s-1}p_s=z_{s-1}\overline{z_{s-1}}+1$ can be written as $p_{s-1}=[q_{s},\overline{q_{s}}]=[\overline{q_{s}},q_{s}]$, since $p_s=1$,  $z_{s-1}=q_s$, and $q_{s}$ and $\overline{q_{s}}$ commute; see Lemma \ref{lemm:LC_IdentNonComm}.  From the equation $z_{s-2}=q_{s-1}p_{s-1}+z_{s-1}$ and Remark \ref{revanti}, it follows that $z_{s-2}=[q_{s-1},\overline{q_{s}},q_{s}]$.  The equation $p_{s-2}p_{s-1}=z_{s-2}\overline{z_{s-2}}+1$ can therefore be written as 
\begin{align*}
p_{s-2}[\overline{q_{s}},q_{s}]=&[q_{s-1},\overline{q_{s}},q_{s}]\overline{[q_{s-1},\overline{q_{s}},q_{s}]}+1&\\
p_{s-2}[\overline{q_{s}},q_{s}]=&[q_{s-1},\overline{q_{s}},q_{s}][\overline{q_{s}},q_{s},\overline{q_{s-1}}]+1& \text{(by Property P--\ref{PPantiauto})}
\end{align*} 
Hence, $p_{s-2}=[q_{s-1},\overline{q_{s}},q_{s},\overline{q_{s-1}}]$ (by Lemma \ref{lemm:LC_IdentNonComm}). Continuing this process, we obtain continuant representations for $p_{s-3}$,$\ldots$,$p_0$. The representation for $p_0=p$ is the quasi-palindromic continuant $[q_1,\overline{q_2},\ldots, q_2,\overline{q_1}]$,
where the central pair
is $q_s,\overline{q_s}$ if $s$ is odd and $\overline{q_s},q_s$ if $s$ is even. Thus, we have a representation of $p=x\overline x+y\overline y$, with $x$ and $y$ being as follows:\begin{align*}
x&=\begin{cases}
[q_1,\overline{q_2},\ldots,\overline{q_{s-1}}, q_s]& \textrm{if $s$ is odd}\\
[q_1,\overline{q_2},\ldots,q_{s-1},\overline{q_s}]& \textrm{if $s$ is even}.
\end{cases}\\
y&=\begin{cases}[q_1,\overline{q_2},\ldots,\overline{q_{s-1}}]& \text{if $s$ is odd}\\
[q_1,\overline{q_2},\ldots,q_{s-1}]& \text{if $s$ is even}.
\end{cases}
\end{align*}
This completes the proof of the four-square theorem.

Consider the following example, where $p_0=431$ and $z_0=54+10i$.
\begin{align*}
431\cdot 7&=(54+10i)(54-10i)+1 &\rightarrow& &54+10i&=(8+i)7+(-2+3i)\\
7\cdot 2&=(-2+3i)(-2-3i)+1 &\rightarrow& &-2+3i&=(-1+i)2+i\\
2\cdot 1&=(i)(-i)+1 &\rightarrow& &i&=i\cdot1
\end{align*}

Hence $(q_1,q_2,q_3)=(8+i,-1+i,i)$, $x=[8+i,-1-i,i]$ and $y=[8+i,-1-i]$. Thus,\begin{align*}
431&=[8+i,-1-i,i,-i,-1+i,8-i]\\
&=[8+i,-1-i,i]\overline{[8+i,-1-i,i]}+[8+i,-1-i]\overline{[8+i,-1-i]}\\
&=(17-5i)(17-5\overline i)+(-6-9i)(-6-9\overline i)\\
&=17^2+5^2+6^2+9^2
\end{align*}

\begin{remark}[Number of representations by the form $x^2+y^2+z^2+u^2$]
The number of representations of positive numbers by this form is given by Jacobi's theorem \cite[Thm.~9.5]{KW}. See also \cite[Sec.~20.12]{HW08} and \cite{Hir87,AlaAlaLem07,AndEkhZei93,SpeWil00}.
\end{remark}

\section{Some quadratic forms representing  integers}\label{otherform}
Using the techniques of Section \ref{4sq}
we may build other forms representing either all positive integers or all integers. Examples of forms representing all positive integers are $x^2-xy+y^2+z^2-zu+u ^2$ and $x^2+3y^2+z^2+3u^2$, while the form $x^2-3y^2+z^2-3u^2$ is an example of a form representing all integers. Some of these results have already appeared in the literature; see, for instance, \cite[Thm.~1.9]{AlaAlaLem07}, \cite[Thm.~12]{AlaAlaWil06}, \cite{Cha08} and \cite[Thm.~13]{HuaOuSpe02}.

\begin{proposition}\label{prop4}
Each positive integer has the form $x^2-xy+y^2+z^2-zu+u ^2$ with $x,y,z,u$ integers.
\end{proposition}
\begin{proof}
Consider the ring $\mathbb{Z}[j]$ of Eisenstein integers, with $j=\exp(2i\pi/3)$, endowed with its natural anti-automorphism. We note that $v^2-vw+w^2$ is the norm of $v+wj$. As in Section~\ref{4sq}, Lemma \ref{prodform} for the case $R=\mathbb Z[j]$ reduces the task to primes. Again, as in Section~\ref{4sq}, every prime $p$ is either of the form $z\overline z$ or divides some $z\overline z+1$, with $z\in \mathbb Z[j]$. See \cite[Prop.~4.7]{Cox89}.

Assume the equation $z\overline z +1 \equiv 0\pmod p$ admits a solution $z_0$ over $\mathbb{Z}[i]$. Then, reasoning as in Section~\ref{4sq}, the division process provides a deterministic algorithm to find  a representation $p=x\overline x+y\overline y$. Here again we reduce $z_0$ modulo $p$ and assume $z_0\overline z_0\le 3 p^2/4$. Thus, we only have to be careful if $p_{s-1}=2$ to avoid the trap $2\cdot 2= (1-j)(1-\overline{j})+1$, where $p_{s-1}=p_s=2$. This problem is avoided by choosing a convenient quotient $q_{s-1}$.
\end{proof}

Next we show an example with the aforementioned trap, that is, where the sequence $p_0,\ldots,p_s$ is not decreasing. Take $p_0=47$ and $z_0=11+7j$, then $94=47\cdot 2=(11+7j)(11+7 \overline j)+1$. Here we have $p_1=2$. The equation $11+7j=q_1p_1+z_1$ with the quotient $q_1=5+4j$ would produce $z_1=1-j$ and $p_2=2$, that is, $2\cdot 2=(1-j)(1-\overline j)+1$. However, with the quotient $q_1=5+3j$, we get $z_1=1+j$ and $p_2=1$, that is, $2\cdot 1=(1+j)(1+\overline j)+1$ and $q_2=1+j$. Hence, $(q_1,q_2)=(5+3j, 1+j)$ and
\begin{align*}
47&=[5+3j,1+\overline j,1+j,5+3\overline j]\\
&=[5+3j, 1+\overline j]\overline{[5+3j, 1+\overline j]}+[5+3j]\overline{[5+3j]}\\
&=(4-2j)(4-2\overline j)+(5+3j)(5+3\overline j)\\
&=4^2-(4)(-2)+(-2)^2+5^2-5\cdot 3+3^2.
\end{align*}

Previous proofs of Proposition \ref{prop4} appeared in \cite[Thm.~12]{AlaAlaWil06}, \cite{Cha08}, and \cite[Thm.~13]{HuaOuSpe02}. The proof in \cite[Thm.~13]{HuaOuSpe02} is perhaps the first elementary proof.

\begin{remark}[Number of representations by the form $x^2-xy+y^2+z^2-zu+u ^2$]
The number of representations is given by Liouville's theorem \cite[Thm.~ 17.3]{KW}. See also \cite[Thm.~12]{AlaAlaWil06}, \cite{Cha08}, and \cite[Thm.~13]{HuaOuSpe02}.
\end{remark}

\begin{corollary}\label{cor:2form}
Every positive integer has the form $x^2+3y^2+z^2+3u ^2$.
\end{corollary}
\begin{proof} By Proposition \ref{prop4} we only need to prove that $x^2-xy+y^2$ has the form $p^2+3q^2$. Indeed,
\begin{enumerate}
\item
If $x$ is even, say $x=2t$, then $x^2-xy+y^2=4t^2-2ty+y^2=(y-t)^2+3t^2$
\item
If $y$ is even, say $y=2t$, then $x^2-xy+y^2=(x-t)^2+3t^2$
\item
If $x$ and $y$ are both odd, then $x^2-xy+y^2=((x+y)/2)^2+3((y-x)/2)^2$
\end{enumerate}
\end{proof}

A proof of Corollary \ref{cor:2form} appeared in \cite[Thm.~1.9]{AlaAlaLem07}.
 
\begin{remark}[Number of representations by the form $x^2+3y^2+z^2+3u ^2$]
The number of representations was stated without proof by Liouville \cite{Lio60,Lio63} and it is proved in \cite[Thm.~1.9]{AlaAlaLem07}.
\end{remark}

\begin{proposition}
Each integer  has the form $x^2-3y^2+z^2-3u ^2$.
\end{proposition}
\begin{proof} This can be proved by reasoning as in Proposition \ref{prop4}. The necessary ring is $\Z[\sqrt3]$ endowed with its natural anti-automorphism.\end{proof}

In the following example we try to represent 19 and $-19$, noticing that $19\cdot 2= 7^2-3\cdot 2^2+1$.
\[
\begin{array}{l r}
19\cdot 2 = (7+\sqrt3)(7-\sqrt 3) +1\\
& q_1=3+\sqrt 3\\
2\cdot 1 =(1+0\sqrt 3)(1-0\sqrt 3)+1\\
&q_2=1+0\sqrt 3.
\end{array}
\]
Hence
\begin{align*}
19&=[3+\sqrt 3,1-0\sqrt3][3-\sqrt 3,1+0\sqrt3]+[3+\sqrt 3][3-\sqrt 3]\\
&=(4+\sqrt3)(4-\sqrt3)+(3+\sqrt3)(3-\sqrt3)\\
&=4^2-3\cdot1^2+3^2-3\cdot1^2.
\end{align*}

Then,  to represent $-19$, we use $-1=1\cdot 1+(1+\sqrt3)(1-\sqrt3)$ and the product formula (Lemma \ref{prodform}) to get
\begin{align*}
-19&=((4+\sqrt3)(1+\sqrt3)+(3+\sqrt3))\overline{(4+\sqrt3)(1+\sqrt3)+(3+\sqrt3)}\\
&+((4+\sqrt3)-(3+\sqrt3)(1-\sqrt3))\overline{(4+\sqrt3)-(3+\sqrt3)(1-\sqrt3)}\\
&=(10+6\sqrt3)(10-6\sqrt3)+(4+3\sqrt3)(4-3\sqrt3)\\
&=10^2-3\cdot6^2+4^2-3\cdot3^2.
\end{align*}

\begin{remark}
If continuants over new rings are considered, the approach presented in Sections \ref{4sq} and \ref{otherform} is likely to provide more quaternary quadratic forms representing either all positive integers or all integers.   \end{remark}

\end{document}